\documentclass[12pt, reqno]{amsart}
\usepackage{amsmath, amsthm, amscd, amsfonts, amssymb, graphicx, color}
\usepackage[bookmarksnumbered, colorlinks, plainpages]{hyperref}

\newtheorem{theorem}{Theorem}[section]
\newtheorem{lemma}[theorem]{Lemma}

\newtheorem{corollary}[theorem]{Corollary}
\theoremstyle{definition}
\newtheorem{definition}[theorem]{Definition}
\newtheorem{example}[theorem]{Example}

\theoremstyle{remark}

\numberwithin{equation}{section}
\newcommand{\supp}{\operatorname{supp}}

\newcommand{\Lip}{\operatorname{Lip}}

\newcommand{\lip}{{\rm lip}}

\newcommand{\lo}{\longrightarrow}
\newcommand{\BC}{\Bbb C}

\newcommand{\BN}{\Bbb N}

\newcommand{\coz}{{\rm coz}}
\begin{document}
\setcounter{page}{1}

\title[]{Jointly separating maps between vector-valued function spaces}

\author[Z. Pourghobadi, M. Najafi Tavani, F. Sady]{Z. Pourghobadi, M. Najafi Tavani, F. Sady}

\address{Department of Mathematics, Faculty of Basic Sciences, Islamic Azad University, Science and Research Branch,  Tehran, Iran}
\email{\textcolor[rgb]{0.00,0.00,0.84}{zibapourghobadi@gmail.com}}

\address{Department of Mathematics, Faculty of Basic Sciences, Islamic Azad University, Islamshahr Branch, Islamshahr, Tehran, Iran}
\email{\textcolor[rgb]{0.00,0.00,0.84}{najafi@iiau.ac.ir}}

 \address{Department of Pure Mathematics, Faculty of Mathematical Sciences, Tarbiat Modares University, Tehran 14115--134, Iran}
 \email{\textcolor[rgb]{0.00,0.00,0.84}{sady@modares.ir}}




\begin{abstract}
Let $X$ and $Y$ be compact Hausdorff spaces, $E$ and $F$ be real or complex Banach spaces, and $A(X,E)$ be a subspace of $C(X,E)$. In this paper we study linear operators $S,T: A(X,E) \lo C(Y,F)$ which are jointly separating, in the sense that $\coz(f) \cap
\coz(g) = \emptyset$ implies that $\coz(Tf) \cap
\coz(Sg)=\emptyset$. Here $\coz(\cdot)$ denotes the cozero set of
a function.  We characterize the general form of such maps between certain class of vector-valued (as well as scalar-valued) spaces of  continuous functions including spaces of  vector-valued Lipschitz functions, absolutely continuous functions and continuously differentiable functions. The results can be applied for a pair $T:A(X) \lo A(X)$ and $S:A(X,E) \lo A(X,E)$ of linear operators, where $A(X)$ is a regular Banach function algebra on $X$, such that $f\cdot g=0$ implies $Tf \cdot Sg=0$, for $f\in A(X)$ and $g\in A(X,E)$. If $T$ and $S$ are jointly separating bijections between Banach algebras of scalar-valued functions  of this class, then they induce a homeomorphism between $X$ and $Y$ and, furthermore,  $T^{-1}$ and  $S^{-1}$ are also jointly separating maps.
\end{abstract}
\maketitle

\section{introduction}
Let $\Bbb K$ denote the field of real or complex numbers. Let $X, Y$ be topological spaces,  and $A(X,E)$, $A(Y,F)$ be spaces of continuous functions on $X$ and $Y$, respectively, taking their values in $\Bbb K$-normed spaces $E$ and $F$. A linear operator $T : A(X,E)\lo A(Y,F)$ is called {\em separating} if for each $f, g
\in A$ we have $\|Tf(\cdot)\| \|Tg(\cdot)\|\equiv 0$ whenever
$\|f(\cdot)\| \|g(\cdot)\|\equiv 0$ or, equivalently,  $\coz(Tf)
\cap \coz (Tg) = \emptyset $ whenever $\coz(f) \cap \coz(g) =
\emptyset $. Here $\coz(\cdot)$ denotes the cozero set of a
function.  A bijective separating map $T$ whose inverse is also separating is called a {\em biseparating} map.

Weighted composition operators are standard examples of separating maps between spaces of functions. The study of separating maps on various spaces of continuous functions (in either of  scalar or vector valued case) has attracted a considerable interest for many years, see for example \cite{abcev, aev,  abn, aj, v, vw,  kn}. The notation of separating maps between spaces of scalar-valued continuous functions defined on compact Hausdorff spaces was introduced in \cite{bnt}. In \cite{j}, K. Jarosz described the general form of a linear separating map from $C(X)$ to $C(Y) $ where $X$ and $Y$ are compact Hausdorff spaces. This result has been generalized in \cite{fh} and \cite{jw} to the locally compact case. Characterization and automatic continuity of separating maps between regular commutative semisimple Banach algebras were studied in \cite{f}.

The vector-valued case has also been considered  in some contexts, such as \cite{gjw, hbn, jr, vw}. In \cite{gjw} authors considered biseparating linear bijections between spaces of vector-valued continuous functions and  showed that for compact Hausdorff spaces $X, Y$ and Banach spaces $E,F$, any biseparating linear map $T$ from  the Banach space $C(X,E)$ of all continuous $E$-valued functions defined on $X$, onto $C(Y, F)$, induces a homeomorphism between $X$ and $Y$. They also described the general form of such maps. Similar results were given in \cite{vw, vvw} for the case where $T$ is a linear biseparating map between little Lipschitz function spaces. Separating linear bijections on spaces of vector-valued absolutely continuous functions defined on compact subsets of the real line were studied in \cite{du}.

In \cite{abcev}  the authors are concerned with continuous bilinear maps $\phi$ from $C^1[0,1] \times C^1[0,1]$ to a Banach space $E$ such that $f\cdot g=0$ implies $\phi(f,g)=0$. A similar problem has been considered in  \cite{aev} for the Banach algebra of  little Lipschitz functions instead of $C^1([0,1])$. Clearly any separating map $T: \mathcal A \lo \mathcal B$  between spaces of scalar-valued functions
 $\mathcal A$ and $\mathcal B$, satisfies this implication
for $\phi(f,g)=Tf \cdot Tg$.

In this paper, we study jointly separating linear maps on vector-valued spaces of functions, having $\mathcal{A}$-module structure for some regular Banach function algebra $\mathcal{A}$. For compact Hausdorff spaces $X,Y$, Banach spaces $E,F$ and a certain subspace $A(X,E)$ of $C(X,E)$ we first study bilinear maps $\phi: A(X,E)\times A(X,E) \lo C(Y)$ of the form $\phi(f,g)=Tf
\cdot Sg$, $f,g \in A(X,E)$, for some linear maps $T,S: A(X,E)\lo
C(Y)$, satisfying the implication
\[ \coz(f)\cap \coz(g)=\emptyset \Rightarrow \phi(f,g)=0.\]
Such linear maps $S,T$ are said to be {\em jointly separating}. Then we study jointly separating maps $T,S: A(X,E) \lo A(Y,F)$ between spaces of vector-valued continuous functions in the sense that $\coz(f)\cap \coz(g)=\emptyset$ implies $\coz(Tf)\cap
\coz(Sg)=\emptyset$. Under some extra conditions,  jointly
separating maps are  continuous and induce some  homeomorphisms between the underlying topological spaces. The results can be applied for a pair  $T:A(X) \lo A(X)$ and $S:A(X,E) \lo A(X,E)$ of linear operators, where $A(X)$ is a regular Banach function algebra on $X$, such that $f\cdot g=0$ implies $Tf \cdot Sg=0$, for $f\in A(X)$ and $g\in A(X,E)$. An automatic continuity property will also be given in finite dimensional case.

\section{preliminaries}

For (real or complex) topological vector spaces $E$ and $F$, we use the notation $L(E,F)$ for  the space of all continuous linear operators  from $E$ to $F$ and we simply write $L(E)$ for $L(E,E)$. For a normed space $E$, $E^*$ is its dual space and $E_w$ stands for the space $E$ equipped with the weak topology.

 For a compact Hausdorff space $X$ and a Banach space $E$ over
$\Bbb K=\Bbb C$ or $\Bbb R$, $C(X,E)$ is the Banach space of all continuous $E$-valued functions on $X$, and $\|\cdot\|_\infty$ denotes the supremum norm on $C(X,E)$. Also $C(X)$ denotes the space of complex-valued continuous functions on $X$. For $f\in C(X,E)$, $\coz(f)$  denotes the cozero set of $f$ and ${\rm supp}(f)=\overline{\coz(f)}$. For $e\in E$, $\tilde{e}$ is the constant function in $C(X,E)$ which sends each point $x \in X$ to $e$. A complex subalgebra $A(X)$ of $C(X)$ is a {\em function algebra} on $X$, if it contains constants and separates the points of $X$. A {\em Banach function algebra} is a function algebra on $X$ which is a Banach algebra with respect to a norm.

A Banach function algebra $A$ on $X$ is {\em regular} if for any point  $\varphi$  and closed subset $F$ in the maximal ideal space of $A$  with $\varphi \notin F$, there exists $f\in A$ such that its Gelfand transform $\widehat{f}$ satisfies $\widehat{f}(\varphi)=1$ and $\widehat{f}=0$ on $F$. If $A$ is  a regular Banach function algebra on $X$, then for any pair $K_1$ and $K_2$ of nonempty disjoint compact subsets of $X$ there exists $f\in A$ with $f(K_1)=\{1\}$ and $f(K_2)= \{0\}$. Moreover, for any compact subset $K$ of $X$ and open cover $\{U_1, ... , U_n\}$ of $K$, there exist $f_1, ... , f_n \in A$ such that $\sum_{i=1}^nf_i(x)=1$ on $K$ and ${\rm supp}(f_i)\subseteq U_i$, for each $i=1,...,n$.



 Let $X,Y$ be compact Hausdorff spaces, $E,F$ be
Banach spaces over $\Bbb K$ and $A(X,E), A(Y,F)$ be subspaces of $C(X,E)$ and $C(Y,F)$, respectively.
\begin{definition}
 We say that  linear operators $T,S: A(X,E)\lo A(Y,F)$ are {\em jointly
 separating  with respect to a point $v^{*}\in F^*$} if $(v^{*}\circ Tf) \cdot
(v^{*}\circ Sg)=0$ for all $f,g \in A(X,E)$ with $\coz (f) \cap
\coz (g) =\emptyset$. We also say that $T$ and $S$ are {\em
jointly separating} if $\coz (Tf) \cap \coz (Sg) =\emptyset$ for all $f,g \in A(X,E)$ with $\coz (f) \cap \coz (g) =\emptyset$.
\end{definition}

We should note that  linear maps $T,S: A(X,E)\longrightarrow A(Y,F)$ are jointly separating if an only if they are jointly separating with respect to all $v^*\in F^*$. The ``only if'' part is trivial. For the ``if part''  assume that $T,S$ are jointly separating with respect to all $v^*\in F^*$. Assume on the contrary that there exists a point $y\in \coz (Tf) \cap \coz (Sg)$ where $f,g \in A(X,E)$  such that $\coz (f) \cap \coz (g) =\emptyset$. Setting $Tf(y)=e_1$ and $Sg(y)=e_2$ we can easily find $v^*\in F^*$ such that $v^*(e_1)\neq 0$ and $v^*(e_2)\neq 0$. Hence $ v^{*}( Tf(y))\, v^{*}( Sg(y))\neq 0$,  a contradiction.

\section{jointly separating maps with values in C(Y)}
Throughout this section we assume that $X,Y$ are compact Hausdorff spaces, $E$ is a Banach space over $\Bbb K$, $A(X,E)\subseteq C(X,E)$ is a subspace and $(A(X),\|\cdot\|_{A(X)})$  is a regular Banach function algebra on $X$ such that $A(X)\cdot A(X,E)
\subseteq A(X,E)$. We should note that   $A(X) \cdot
A(X,E)=\{f\cdot h: f\in A(X), h\in A(X,E)\}$, where  for $f\in A(X)$ and $h\in A(X,E)$, $f\cdot h$ is an element in $C(X,E)$ defined by $f\cdot h(x)=f(x)h(x)$, $x\in X$. We also assume that $T, S : A(X, E) \lo C(Y)$ are jointly separating linear maps.

\begin{definition}\label{def2}
Let $Y_1\subseteq Y$ be  defined by $$Y_1=\big(\cup_{f \in A(X,E)} \coz (Tf)\big) \cap \big(\cup_{f \in A(X,E)} \coz (Sf )\big).$$ We call a  point $x \in X$ a {\em support point} of $y \in Y_1$ if for each neighborhood $V$ of $x$ there exists $f \in A(X,E)$ such that $\coz(f) \subseteq V$ and $$ Tf(y) \neq 0 \ \ {\rm or} \ \   Sf(y) \neq 0.$$
\end{definition}

\begin{lemma}\label{lem1}
The set of support points of  each $y \in Y_1$ is non-empty.
\end{lemma}
\begin{proof}
Suppose, on the contrary, that for each $x \in X$ there exists a neighborhood $V_x$ of $x$ such that for each $f \in A(X,E)$ with $\coz(f) \subseteq V_x$ we have $Tf(y))= Sf(y) = 0$. Since $X
\subseteq \cup_{x\in X} V_x$,  there exist finitely many points
$x_1 , ..., x_n\in X$ such that  $X \subseteq \cup_{i=1}^n V_{x_i}$. By the regularity of  $A(X)$, there exist functions $f_1, ..., f_n\in A(X)$ such that for each $1\leq i \leq n$ we have $\coz(f_i) \subseteq V_{x_i}$ and $\sum_{i = 1}^n f_i =1$ on $X$. Clearly for each $f \in A(X,E)$ we have $\sum_{i = 1}^n f_i \cdot f =f$. Furthermore, since $\coz(f_i\cdot f) \subseteq V_{x_i}$, we have $T(f_i \cdot f)(y)=  S(f_i\cdot f)(y) = 0$. Thus $ Tf(y)= Sf(y) = 0$ for all $f \in A(X,E)$, which contradicts to the fact that $y \in Y_1$.
\end{proof}

\begin{lemma}\label{lem2}
Let $x \in X$ be a support point of $y \in Y_1$. Then for each neighborhood $V$ of $x$ there exist $f, g \in A(X,E)$ such that $\coz(f)\cup \coz(g)\subseteq V$ and, furthermore, $ Tf(y)\neq 0$ and $  Sg(y)\neq 0$.
\end{lemma}
\begin{proof}
We need only to find $f\in A(X,E)$ such that $\coz(f)\subseteq V$ and $Tf(y)\neq 0$, since $g$ is given in a similar manner. Assume on the contrary that there exists a neighborhood $V$ of $x$ such that for each $f \in A(X,E)$ with $\coz(f) \subseteq V$ we have $Tf(y)=0$. Choose a neighborhood $U$ of $x$ with $\overline{U}\subseteq V$. Since $x$ is a support point of $y$, there exists a function $f \in A(X,E)$ such that $\coz(f)
\subseteq U$ and $$  Tf(y) \neq 0 \ \ {\rm or} \ \  Sf(y) \neq
0.$$ Since by hypothesis, $ Tf(y)=0$ we get $Sf(y) \neq 0$. We note that for each  $h \in A(X,E)$ if $\coz(h) \subseteq V$ then, by hypothesis, $  Th(y)=0$ and if $\coz(h) \subseteq X \setminus
\overline{U}$ then, $\coz(f) \cap \coz(h)=\emptyset$ and
consequently,  we have again $  Th(y)=0$.

Now since  $X\subseteq V \cup (X\backslash \overline{U})$, there exists $f_1,f_2 \in A(X)$ such that $f_1+f_2 =1$ on $X$ and, furthermore,  $\coz(f_1) \subseteq V$ and $\coz(f_2) \subseteq X\backslash \overline{U}$.  Then for each $f \in A(X,E)$ we have $f\cdot f_1 + f \cdot f_2=f$ and $T(f_i \cdot f)(y)= 0$, for $i=1,2$,  since $\coz(f_1\cdot f) \subseteq V$ and $\coz(f_2\cdot f) \subseteq X\backslash \overline{U}$. But this clearly implies that $ Tf(y)= 0$ for all $f \in A(X,E)$, which contradicts to the fact that $y \in Y_1$.
\end{proof}

\begin{corollary}\label{coro1}
Let $x \in X$ be a support point of $y \in Y_1$ and $f \in A(X,E)$. Then $x \in {\rm supp}(f)$ whenever $ Tf(y)\neq 0$ or $Sf(y)\neq 0$.
\end{corollary}
\begin{proof}
If $x \notin {\rm supp}(f)$, then there exist neighborhoods $V$ and $U$ of $x$ and ${\rm supp}(f)$, respectively, with disjoint closures. So, by Lemma \ref{lem2}, there exist functions $h, g \in A(X,E)$ with  $\coz(h) \cup \coz(g) \subseteq  V$, such that  $ Th(y)\neq 0$  and $ Sg(y)\neq 0$. Now if $ Tf(y)\neq 0$, then $ Tf(y)\cdot  Sg(y) \neq 0$, while $\coz (f) \cap \coz (g) =\emptyset$, a contradiction. Similarly, if $Sf(y)\neq 0$, then $Th(y)\cdot Sf(y) \neq 0$, while $\coz (f) \cap \coz (h) =\emptyset$. Thus in both cases we get a contradiction.
\end{proof}
\begin{corollary}\label{coro2}
The set of support points of each $y \in Y_1$ is a singleton.
\end{corollary}
\begin{proof}
Let $x_1$ and $x_2$ be  distinct support points of $y \in Y_1$. Then there exist neighborhoods $V_{1}$ and $V_{2}$ of  $x_1$ and $x_2$, respectively with disjoint closures. By definition, there exists a function $f \in A(X,E)$ such that $\coz(f) \subseteq V_{1}$ and $$ Tf(y) \neq 0 \;\; {\rm  or } \; \;  Sf(y) \neq 0.$$ But, by Corollary \ref{coro1}, we have $x_2 \in {\rm supp}(f)
\subseteq \overline{V}_{1}$ which is a contradiction.
\end{proof}
Corollary \ref{coro2} allows us to  define a function $\phi: Y_1\longrightarrow X$ with $\phi(y)=x$ where $x$ is the unique support point of $y\in Y_1$. We call  $\phi$ the support map of $T$ and $S$.

The following corollary is a straightforward consequence  of Corollary \ref{coro1}.
\begin{corollary} \label{coro3}
Let $f \in A(X,E)$ and $U \subseteq X$ be open. Then $f|_U=0$ implies that  $ Tf|_{\phi^{-1}(U)}=0$ and $Sf|_{\phi^{-1}(U)}=0$.
\end{corollary}

\begin{lemma} \label{lem3}
The support map $\phi$ is continuous.
\end{lemma}
\begin{proof}
Let $\{y_\alpha\}$ be a net in $Y_1$ converging to $y_0 \in Y_1$. Since $X$ is compact, the net $\{\phi(y_\alpha)\}$ has a convergent subnet, which is denoted again  by $\{\phi(y_\alpha)\}$, to a point $x_0\in X$. Suppose, on the contrary, that $\phi(y_0) \neq x_0$. Then there exist neighborhoods $V_{0}$ and  $V_{1}$ of $x_0$ and $\phi(y_0)$, respectively, with disjoint closures. By definition, there exists $f \in A(X,E)$ such that $\coz(f) \subseteq V_{1}$ and $$ Tf(y_0) \neq 0 \ \ or \ \  Sf(y_0) \neq 0.$$ Without  loss of generality we may assume that $ Tf(y_0) \neq 0 $. Since $Tf$ is continuous, there exists a neighborhood $W_{0}$ of $y_0$ such that  $ Tf(y) \neq 0 $ for each $y \in W_{0}$. On the other hand, since $y_\alpha \to y_0$, there exists $\beta$ such that $y_\alpha \in W_{0}$ for all  $\alpha
> \beta$, which implies that $ Tf(y_\alpha) \neq 0 $. Now
by Corollary \ref{coro1} we have   $\phi(y_\alpha) \in {\rm supp}(f) \subseteq \overline{V}_{1} $ for each $\alpha
> \beta$, which contradicts to the fact that $\phi(y_\alpha)
\to x_0$ and $\overline{V}_{0}\cap \overline{V}_{1} = \emptyset$.
This implies that $\phi$ is continuous.
\end{proof}
\begin{definition}\label{def3}
Let $\|\cdot\|$ be a norm on  $A(X,E)$ with $\|\cdot \|\ge \|\cdot
\|_\infty$ . We define the property (P) on $A(X,E)$ as follows

\begin{quote}
for each $f\in A(X,E)$ and $x\in X$ with $f(x)=0$ there exists a sequence $\{f_n\}$  in $A(X,E)$, such that each  $f_n$  vanishes  on a neighborhood  of $x$, and, furthermore,  $\|f_n-f\|\to 0$ as $n \to \infty$. \hfill (P)
\end{quote}
\end{definition}

If $(A(X,E),\|\cdot\|)$ satisfies  Ditkin's condition with respect to $A(X)$, in the sense that  for every $f \in A(X,E)$ and $x \in X$ with $f(x)=0$, there exists a sequence $\{f_n\}$  of functions in $A(X)$ such that each  $f_n$ vanishes on a neighborhood  of $x$ and, furthermore, $\lim_n\|f_n\cdot f-f\|=0$, then it has (P) property. Clearly any Ditkin Banach function algebra on $X$ satisfies the Ditkin's condition with respect to itself. For some examples of subspaces $A(X,E)$  satisfying Ditkin's condition or having (P) property, see Example \ref{exam}.

For each $x\in X$, let $\delta_x: A(X,E) \longrightarrow E$ be defined by $\delta_x(f)=f(x)$ for each $f \in A(X,E)$ .

The next lemma is an immediate consequence of Corollary \ref{coro3}.
\begin{lemma}\label{lem4}
Let $(A(X,E),\|\cdot\|_\infty)$ have (P) property.  Let $y \in Y_1$ such that the linear functional $\delta_y \circ T$ (resp. $\delta_y \circ S$) is continuous on $A(X,E)$. Then for each $f\in A(X,E)$, with $f(\phi(y))=0$ we have $Tf(y)=0$ (resp. $Sf(y)=0$).
\end{lemma}
Motivated by the above lemma  we define two subsets $$Y_{cT}=\{y\in Y_1:  \delta_y \circ T \; {\rm  is \; \|.\|_\infty-continuous \; on \;}  A(X,E)  \} $$ and $$ Y_{cS}=\{y\in Y_1:   \delta_y \circ S \; {\rm is \;
\|.\|_\infty-continuous \; on \; } A(X,E) \}. $$ of $Y_1$.
\begin{lemma}\label{lem5}
If $(A(X, E),\|\cdot\|_\infty)$ has (P) property and $A(X,E)$ contains constants, then there exist maps $\Lambda^T: Y_{cT}
\longrightarrow E^*$, $y \mapsto \Lambda^T_y$ and $\Lambda^S :
Y_{cS} \longrightarrow E^*$, $y \mapsto \Lambda^S_y$ such that
\begin{equation}\label{eq1}
Tf(y)=\Lambda^{T}_y(f(\phi(y)) \; \;\;\;\; (f \in A(X,E), y \in Y_{cT}),
\end{equation}
 and
\begin{equation}\label{eq2}
Sf(y)=\Lambda^{S}_y(f(\phi(y)) \;\;\;\;\; (f\in A(X,E), y \in Y_{cS}).
\end{equation}
\end{lemma}
\begin{proof}
For each $y \in Y$, we define the linear functional $\Lambda^{T}_y : E \longrightarrow \BC$ by $\Lambda^{T}_y(e)=T(\tilde{e})(y)$, for each $e \in E$.  Clearly for $y \in Y_{cT}$ we have $\Lambda^{T}_y \in E^*$. Now  for each $f \in A(X,E)$ and $y \in Y_{cT}$ set $g=f-\widetilde {f(\phi(y))}$. Since $g(\phi_(y))=0$, it follows from Lemma \ref{lem4} that  $Tg(y)=0$ which implies that $Tf(y)=\Lambda^{T}_y(f(\phi(y))$.

For the second part,   it suffices that for each $y \in Y_{cS}$ we define the continuous linear functional $\Lambda^{S}_y : E
\longrightarrow \BC$ by $\Lambda^{S}_y(e)=S(\tilde{e})(y)$.
\end{proof}
\begin{lemma}\label{lem6}
Let $Y_0=Y \setminus Y_1$. Then under the hypotheses of the above lemma $Y_{cT} \cup Y_0$ and  $Y_{cS} \cup Y_0$ are closed subsets of $Y$.
\end{lemma}
\begin{proof}
Let $y \in Y_{cT}$ and  $\Lambda^{T}_y$ be as in Lemma \ref{lem5}. For each $e \in E$ we have $$|\Lambda^{T}_y(e)|=|T\tilde{e} (y)|\leq \|T\tilde{e}\|_\infty,$$ and consequently  $\sup\{|\Lambda^{T}_y(e)|: y \in Y_{cT}\} \leq
\|T\tilde{e}\|_\infty < \infty$. Hence, by the Banach-Steinhaus
Theorem, there exists $M>0$ such that $$|\Lambda^{T}_y(e)|\leq M \|e\|;  \;\; \;\;\;\;\;\;\;\; ( y \in Y_{cT}  ,  e \in E).$$ This implies that  for each $y \in Y_{cT}$ and $f \in A(X, E)$
\begin{equation}\label{eq3}
|Tf (y)| = |\Lambda^{T}_y(f(\phi(y))|\leq M \|f(\phi(y))\| \leq M
\|f\|_\infty.
\end{equation}
Now let $y\in \overline{Y_{cT} \cup Y_0}$. If $y \in Y_0$ then clearly $y\in Y_{cT} \cup Y_0$. Otherwise, $y\in Y_1$ and since $Y_1$ is open in $Y$,  for each neighborhood $U$ of $y$,  $U\cap Y_1$ has a nonempty intersection with $Y_{cT} \cup Y_0$ and consequently $U\cap Y_{cT}\neq \emptyset$. For any neighborhood $U$ of $y$ choose $y_U\in U\cap Y_{cT}$. By (\ref{eq3}), we have $|Tf (y_U)| \leq M \|f\|_\infty$ for each $f \in A(X, E)$. Consequently $|Tf (y)| \leq M \|f\|_\infty,$ since $U$ is arbitrary. This clearly implies that $\delta_y \circ T$ is continuous on $A(X,E)$, that is, $ y \in Y_{cT}$, as desired. Similarly we can show  that $Y_{cS} \cup Y_0$ is closed in $Y$.
\end{proof}
\begin{theorem}\label{main1}
Let $A(X,E)$ be a subspace of $C(X,E)$ containing constants, $A(X)$ be a regular Banach function algebra on $X$  such that $A(X) \cdot A(X,E) \subseteq A(X,E)$. Assume, furthermore, that $(A(X,E), \|\cdot\|_\infty)$ has (P) property. Then for any pair $T,S: A(X,E) \longrightarrow C(Y)$ of jointly separating linear operators there exist  a partition $\{Y_0, Y_{c}, Y_{d}\}$ of $Y$ where $Y_0$ is closed in $Y$ and $Y_{d}$ is open in $Y$, a continuous map $\phi: Y_c \cup Y_{d} \longrightarrow X$ and also maps $\Lambda^{T}, \Lambda^{S}:Y_{c} \longrightarrow E^*$ such that $$Tf(y)=\Lambda^{T}_y(f(\phi(y)) \;\;\;\;\;\; (f\in A(X,E), y\in Y_c) $$ and $$\ Sf(y)=\Lambda^{S}_y(f(\phi(y))\;\;\;\;\; (f\in A(X,E), y \in Y_c). $$ In particular, if $T$ and $S$ are continuous, then $Y_d=\emptyset$ and $Y_c= \big(\cup_{f \in A(X,E)} \coz (Tf)\big) \cap
\big(\cup_{f \in A(X,E)} \coz (Sf )\big)$.
\end{theorem}
\begin{proof}
Let $Y_1$, $Y_0$, $Y_{cT}$ and $Y_{cS}$ be subsets of $Y$ defined above. We put  $Y_{c}= Y_{cT} \cap Y_{cS}$ and $Y_{d}= Y_1
\setminus Y_{c}$. Let also $\phi: Y_1 \longrightarrow X$, and maps
$\Lambda^T, \Lambda^S: Y_c \longrightarrow E^*$ be defined as before.

 Clearly $Y_0$ is closed and,  by Lemma \ref{lem6}, $Y_{d}$
is open in $Y$.

By Lemmas \ref{lem3} and \ref{lem5}  we have $Tf(y)=\Lambda^{T}_y(f(\phi(y))$ and $Sf(y)=\Lambda^{S}_y(f(\phi(y))$ for all $f\in A(X,E)$ and $y \in Y_c$, as desired.
\end{proof}

\begin{theorem}\label{main1.1}
Let $E$ be finite dimensional, $A(X,E)$ be a subspace of $C(X,E)$ containing constants, $(A(X),\|\cdot\|_{A(X)})$ be a regular Banach function algebra on $X$ such that $A(X) \cdot A(X,E)
\subseteq A(X,E)$. Assume, furthermore, that there exists a
complete norm  $\|\cdot \|$ on $A(X,E)$ with $\|\cdot\|\ge \|\cdot
\|_\infty$ such that $(A(X,E),\|\cdot\|)$ has (P) property and
$\|f \cdot g\|\le \|f\|_{A(X)} \|g\|$ for all  $f\in A(X)$ and $g\in A(X,E)$.  Then
\begin{itemize}
\item[(i)] $\phi( Y_d)$ is finite and consists of  limit points of
$X$;

\item[(ii)] The maps $\Lambda^{T}$, $\Lambda^{S}: Y_c \lo E^*$ are
continuous.
\end{itemize}
\end{theorem}
\begin{proof}
 (i) The proof is a minor modification of \cite{f}.

We first note that, by hypothesis, $(A(X,E),\|\cdot\|_\infty)$ also has  (P) property. Hence we have  the equalities, describing $T$ and $S$ on $Y_c$, in Theorem \ref{main1}. On the other hand since $E$ is assumed to be finite dimensional, for each $y \in Y_1$ the linear functionals $\Lambda^{T}_y, \Lambda^{S}_y : E
\longrightarrow \Bbb C$ defined by
$\Lambda^{T}_y(e)=T(\widetilde{e})(y)$ and $\Lambda^{S}_y(e)=S(\widetilde{e})(y)$, $e\in E$,  are both continuous. Thus for each $y \in Y_d$ there exists $f\in A(X,E)$ such that either $$Tf(y) \neq \Lambda^{T}_y(f(\phi(y)) \;\;\; {\rm or}\; \;\; Sf(y)\neq\Lambda^{S}_y(f(\phi)).$$ Now suppose, on the contrary, that the set $\phi(Y_d)=\phi_(Y_1\setminus Y_{cT}) \cup \phi(Y_1\setminus Y_{cS})$ is infinite. Without loss of generality we may assume that $\phi( Y_1\setminus Y_{cT})$ is infinite. Hence there exists a sequence $\{y_n\}$ in $Y_1\setminus Y_{cT}$ such that $\{\phi( y_n)\}$ is a sequence of distinct points in $\phi( Y_1\setminus Y_{cT})$. Since $X$ is compact, by passing through a subsequence if necessary, we find a sequence $\{U_n\}$ of pairwise disjoint open subsets of $X$ such that, for each $n\in \BN$,  $\phi( y_n)
\in U_n$. Given $n\in \BN$ let $W_n$ be an open neighborhood of
$\phi( y_n)$ such that $\overline{W}_n \subseteq U_n$. Consider, by regularity assumption,  a function $s_n\in A(X)$ such that $s_n=1$ on $\overline{W}_n$ and $\coz(s_n)\subseteq U_n$. Since $y_n \notin Y_{cT}$, there exists $f_n \in A(X,E)$ with $$Tf_n(y_n) \neq \Lambda^{T}_{y_n}(f_n(\phi(y_n)).$$ Put $g_n=f_n-\widetilde{f_n(\phi(y_n)}$.  Then $g_n \in A(X,E)$, $g_n(\phi(y_n)=0$ and  $Tg_n(y_n)\neq 0$. Replacing $g_n$ by an appropriate multiple of $g_n$, we can assume that $|Tg_n(y_n)|
> n$. Since $(A(X,E), \|\cdot \|)$ has (P) property, for each $n\in \Bbb N$ there exist a
 function $h_n \in A(X,E)$ such that $h_n =0 $ on a neighborhood
 $V_n$ of $\phi(y_n)$  and $\|h_n -g_n\| <
\frac{1}{n^2(\|s_n\|_{A(X)}+1)}$. Set $\psi_n=s_n \cdot
(g_n-h_n)$. Then  $\psi_n=g_n$ on $V_n \cap W_n$ and  it follows from Corollary \ref{coro3} that $T\psi_n(y_n)=Tg_n(y_n)$. We note that $\|\psi_n\|\leq \|s_n\|_{A(X)}\|h_n -g_n\|<\frac{1}{n^2}$ for all $n\in \BN$. Since $(A(X,E),
\|\cdot\|)$ is a Banach space, the function $ \psi=
\Sigma_{n=1}^{\infty}\psi_n$ is an element of $A(X,E)$. Clearly,
for each $n\in \BN$ we have $\coz(\psi_n)\subseteq U_n$ and, since $\{U_i\}$ is  pairwise disjoint, we get $\psi=\psi_n$ on $U_n$. Consequently $$|T\psi(y_n)|=|T\psi_n(y_n)|=|Tg_n(y_n)|>n,$$ which is a contradiction since $T\psi$ is a continuous function on the compact space $Y$. Hence $\phi(Y_d)$ is finite.

Now suppose that $\phi(y)$ is an isolated point of $X$ for some $y
\in Y_1$. Then  $V=\{\phi(y)\}$ is an open neighborhood of
$\phi(y)$ and, for each $f \in A(X,E)$, if we set $g=\widetilde{f(\phi(y)}$ then $f=g$ on $V$, and, by Corollary
\ref{coro3}
$$Tf(y)=Tg(y)=T(\widetilde{f(\phi(y)})(y)=\Lambda^{T}_x(f(\phi(y))$$ and $$Sf(y)=Sg(y)=S(\widetilde{f(\phi(y)})(y)=\Lambda^{S}_y(f(\phi(y)).$$ Thus $y\in Y_{cT} \cap Y_{cS}$, since $\Lambda^{T}_y$ and $\Lambda^{S}_y$ are continuous.

(ii) Let $\{y_\alpha\}$ be a net in $Y_1$ converging to a point $y_0\in Y_1$. Fixing $e\in E$, by definition of $\Lambda^{T}_y$ and continuity of $T\tilde{e}$, the net $\{\Lambda^{T}_{y_\alpha}(e)\}$ converges to $\Lambda^{T}_{y_0}(e)$. Since $E$ is finite dimensional this easily implies that $\|\Lambda^{T}_{y_\alpha}- \Lambda^{T}_{y_0}\| \to 0$ in $E^*$, that is $\Lambda^{T}$ is continuous.  Continuity of $\Lambda^{S}$ can be proven  by a similar argument.
\end{proof}

\section{jointly separating maps, vector- valued case}
 In this section, we consider the general case and study  jointly separating linear maps
 between spaces of vector-valued functions.

Let  $X,Y$ be compact Hausdorff spaces, $E,F$ be  Banach spaces over $\Bbb K=\Bbb R$ or $\Bbb C$, and $A(X,E)$ be a subspace of $C(X,E)$. For each linear operator $T: A(X,E)\longrightarrow C(Y,F)$ and each $v^*\in F^*$, let  the linear operator $v^*\circ  T: A(X,E) \longrightarrow C(Y)$ be defined by $v^*\circ T(f)=v^*\circ Tf$. Then, as we noted before, a pair $T,S: A(X,E)\longrightarrow C(Y,F)$ of linear operators is jointly separating if and only if for each $v^*\in F^*$, the pair $v^*\circ T, v^*\circ S: A(X,E) \longrightarrow C(Y)$ is  jointly separating, i.e., $T, S$ are jointly separating with respect to each $v^*\in F^*$. Clearly this is also equivalent to say that for each $v^*, w^* \in F^*$, the pair $v^*\circ T, w^*\circ S: A(X,E)
\longrightarrow C(Y)$ is jointly separating.

\begin{definition}\label{def4}
For a pair of jointly separating linear maps $T,S: A(X,E) \longrightarrow C(Y,F)$ we consider the following subsets of $Y$
\[ Y_1=\big(\cup_{f \in A(X,E)} \coz (Tf )\big) \cap \big(\cup_{f \in A(X,E)}
\coz (Sf)\big),\]
\[Y_c=\{y\in Y_1: \delta_y\circ T, \delta_y \circ S: A(X,E) \lo
F_w \; {\rm are} \; \|.\|_\infty{\rm -continuous} \}.\] We also put  $Y_0=Y\backslash Y_1$ and $Y_d=Y_1 \backslash Y_c$.
\end{definition}
Then $\{Y_0,Y_c,Y_d\}$ is a partition of $Y$ and next  lemma is straightforward.
\begin{lemma}
For each $v^*\in F^*$ let $\{Y^{v^*}_0, Y^{v^*}_c, Y^{v^*}_d\}$ be the partition given in the previous section for jointly separating maps $v^*\circ T, v^*\circ S: A(X,E) \to C(Y)$. Then

{\rm (i)} $Y_1= \cup_{v^*\in F^*}Y_1^{v^*}$, where $Y_1^{v^*}=Y\backslash Y_0^{v^*}$ and $Y_0=\cap_{v^*\in F^*}Y_0^{v^*}$.

{\rm (ii)} $Y_1$ is an open subset of $Y$.

{\rm (iii)} If the ranges of $T$ and $S$ contain  constant functions, then  $Y_1=Y$ and   $Y_c= \cap_{v^*\in F^*} Y^{v^*}_c  $. If, in addition, $T$ and $S$ are $\|.\|_\infty$--continuous, then $Y_c= Y$.
\end{lemma}

\begin{theorem}\label{main2}
Let $A(X,E)$ be a subspace of $C(X,E)$ containing constants, $A(X)$ be a regular Banach function algebra on $X$ such that $A(X)\cdot A(X,E) \subseteq A(X,E)$ and let $(A(X,E),
\|\cdot\|_\infty)$ have (P) property. Let $T,S: A(X,E)
\longrightarrow C(Y,F)$ be jointly separating linear operators.
Then there exist a  continuous map $\phi: Y_c \cup Y_d \longrightarrow X$, and also maps  $\Lambda^T, \Lambda^S: Y_c \to L(E, F_w)$, such that for each $f\in A(X,E)$ and each $ y \in Y_c$
\[ Tf(y)=\Lambda^T_y(f(\phi(y))\] 
and
\[Sf(y)=\Lambda^S_y(f(\phi(y)).\] 
In particular, if $T$ and $S$ are continuous, then  $Y_d=\emptyset$, $Y_c=Y_1$
and $\Lambda^T$ and $\Lambda^S$ take their values in $L(E,F)$. If, in addition the ranges of $T$ and $S$ contain constants, then $Y_c=Y$.
\end{theorem}
\begin{proof}
Let $Y_0$, $Y_1$, $Y_c$ and $Y_d$ be as in Definition \ref{def4}.

For each $v^*\in F^*$, let  $\{Y_0^{v^*}, Y^{v^*}_{c}, Y^{v^*}_{d}\}$ be the partition given in Theorem \ref{main1} for jointly separating maps $v^*\circ T, v^*\circ S: A(X,E) \longrightarrow C(Y)$. Hence $Y^{v^*}_0$ is closed in $Y$ and $Y^{v^*}_{d}$ is open in $Y$, and there exist a continuous map $\phi_{v^*}: Y^{v^*}_c \cup Y^{v^*}_{d} \longrightarrow X$ and also maps $\Lambda^{v^*,T}, \Lambda^{v^*,S}:Y^{v^*}_{c} \longrightarrow E^*$,   such that $$v^*(Tf(y))=\Lambda^{v^*,T}_y(f(\phi_{v^*}(y)) \;\;\;\;\;\; (f\in A(X,E), y\in Y^{v^*}_c) $$ and $$v^*(Sf(y))=\Lambda^{v^*S}_y(f(\phi_{v^*}(y))\;\;\;\;\; (f\in A(X,E), y \in Y^{v^*}_c). $$

The proof will be given through the following steps:
\vspace*{.2cm}

{\bf Step I.} For each $v_1^*, v_2^*\in F^*$ we have $\phi_{v_1^*} =\phi_{v_2^*}$ on $Y_1^{v_1^*}\cap Y_1^{v_2^*}$.

 Assume, on the contrary, that there exists $y \in Y_1^{v_1^*}\cap Y_1^{v_2^*}$ such
that $\phi_{v_1^*} (y)\neq \phi_{v_2^*}(y)$. Then choosing disjoint neighborhoods $U_i$ of $\phi_{v_i^*}(y)$, $i=1,2$, we can find, by Lemma \ref{lem2},  a function $f_1 \in A(X,E)$ such that $\coz(f_1)\subseteq U_1$ and $v_1^*(Tf_1(y)) \neq 0$ , and  similarly a function $f_2\in A(X,E)$ such that $\coz(f_2) \subseteq  U_2$  and $v_2^*(Tf_2(y))\neq 0$. Hence $v_1^*(Tf_1(y)) \cdot v_2^*(Sf_2(y))\neq 0$, which is a contradiction since $\coz(f_1)\cap \coz(f_2) =\emptyset$ and $T$ and $S$ are jointly separating.

\vspace*{.1cm}

 By the above step,  the map $\phi: Y_1
\longrightarrow X$ defined by $\phi(y)= \phi_{v^*}(y)$, where $v^*\in F^*$ such that $y
\in Y_1^{v^*}$, is well-defined.

\vspace*{.2cm}

{\bf Step II.} The map $\phi: Y_1 \lo X$  is continuous.

Let $\{y_\alpha\}$ be a net in $Y_1$ converging to a point  $y_0
\in Y_1$ such that $\{\phi(y_\alpha)\}$ converges to a point $x_0
\in X$ with $x_0 \neq \phi(y_0)$. Consider neighborhoods $V_1$ and
$V_2$ of $x_0$ and $\phi(y_0)$, respectively, with disjoint closures. Since $y_0 \in Y_1$, there exists $v^*\in F^*$ such that $y_0\in Y_1^{v^*}$ and consequently $\phi(y_0)= \phi_{v^*}(y_0)$. By Lemma \ref{lem2} there exists a function $f\in A(X,E)$ such that $\coz(f) \subseteq V_2$ and $ v^{*}( Tf(y_0)) \neq 0 $. Since $\{y_\alpha\}$ converges to $y_0$ and $\{\phi(y_\alpha)\}$ converges to $x_0$, it follows from continuity of $ v^{*} \circ Tf$  that $ v^{*}( Tf(y_{\alpha_0})) \neq 0$ for some $\alpha_0$ with $\phi(y_{\alpha_0}) \in V_1$. Choose  $w^*\in F^*$ with $\phi(y_{\alpha_0})= \phi_{w^*}(y_{\alpha_0})$. Using Lemma
\ref{lem2} once again, we can find $g \in A(X,E)$ with $\coz(g)
\subseteq V_1$ and $w^{*}( Sg(y_{\alpha_0})) \neq 0 $. That is
$\coz(f)\cap \coz(g)=\emptyset$ and $ v^{*}( Tf(y_{\alpha_0})) w^{*}( Sg(y_{\alpha_0})) \neq 0 $, which is a contradiction. So $\phi$ is continuous.

\vspace*{.2cm}

For each $y \in Y_c$,  we define the linear operators $\Lambda^{T}_y, \Lambda^{S}_y : E \longrightarrow F$ by $\Lambda^{T}_y(e)=T\tilde{e} (y)$ and $\Lambda^{S}_y(e)=S\tilde{e} (y)$.
It follows from definition of $Y_c$  that $\Lambda^T_y, \Lambda^S_y \in L(E,F_w)$.  

{\bf Step III.} For each $f\in A(X,E)$ and $y\in Y_c$ 

\begin{equation}\label{eq4}
Tf(y)=\Lambda^T_y(f(\phi(y))
\end{equation}
and
\begin{equation}\label{eq5}
Sf(y)=\Lambda^S_y(f(\phi(y)).
\end{equation}
Let $y\in Y_c$. To prove (\ref{eq4}) it suffices to show that for each $w^*\in F^*$
\begin{align}\label{eq7}
w^*(Tf(y))=w^* ( \Lambda^{T}_y(f(\phi(y))),  \ \ \  (f\in A(X,E))
\end{align}
Since $y\in Y_c\subseteq Y_1$ it follows that there exists $v^*\in F^*$ such that $y\in  Y_1^{v^*}$. Thus $\phi(y)=\phi_{v^*}(y)$.  

Let $w^*\in F^*$ be given. If $y\in  Y_1^{w^*}$ then we have $y\in  Y_c^{w^*}$, since $y\in Y_c$. Clearly  $\Lambda^{w^*,T}_y=w^*\circ \Lambda^T_y$. Hence    
\[w^*(Tf(y))=\Lambda^{w^{*},T}_y(f(\phi_{w^*}(y))= w^* \circ
\Lambda^{T}_y(f(\phi_{w^*}(y))=w^* ( \Lambda^{T}_y(f(\phi(y))),\]
for all  $f\in A(X,E)$, as desired. Now if  $y\notin  Y_1^{w^*}$, then  $y\in  Y_0^{w^*}$ which implies that one of the linear functionals  $w^*\circ \delta_y\circ T$ or  $w^*\circ \delta_y\circ S$ on $A(X,E)$ is equal to zero. Clearly  (\ref{eq7}) holds if $w^*\circ \delta_y\circ T\equiv 0$. So assume that $w^*\circ \delta_y\circ T\neq 0$. 
To prove (\ref{eq7}) we show that for each  $g\in A(X,E)$  with $g(\phi(y))=0$ we have $w^*(Tg(y))=0$. Since $A(X,E)$ has (P) property and $w^*\circ \delta_y\circ T$ is continuous on $A(X,E)$  we may assume that $g$ vanishes on a neighborhood of $\phi(y)$.  As it was noted before, $\phi(y)=\phi_{v^*}(y)$, that is $\phi(y)$ is a support point for $y$ (as an element of $Y^{v^*}_1$ ) for jointly separating maps $v^*\circ T$ and $v^*\circ S$. Hence, by Lemma  \ref{lem2}, there exists $h \in A(X,E)$ such that $\coz(g) \cap \coz(h)=\varnothing$ and $v^*(Sh(y))\neq 0$. Hence $w^*(Tg(y))= 0$, since $T$ and $S$ are jointly separating. This implies that if  $g\in A(X,E)$  and  $g(\phi(y))=0$,  then $w^*(Tg(y))= 0$.  Now let $f\in A(X,E)$ be arbitrary and put $g=f-\widetilde{f(\phi(y)})$. Since $g(\phi(y))=0$ it follows that $w^*(Tg(y))=0$, that is  $w^*(Tf(y))=w^* (T(\widetilde{f(\phi(y)})(y))=w^* ( \Lambda^{T}_y(f(\phi(y)))$, as desired.

The above argument  shows that (\ref{eq7}) holds for each $w^*\in F^*$ and consequently 
\[Tf(y)=\Lambda^T_y(f(\phi(y)) \;\;\;\;\;\;\; (f \in A(X,E), y\in Y_c).\]
Applying the same argument  with $S$ instead of $T$ we conclude that
\[Sf(y)=\Lambda^S_y(f(\phi(y)) \;\;\;\;\;\;\; (f \in A(X,E), y\in Y_c).\]

 This completes the proof of the first part. The second part is easily verified.
\end{proof}

\begin{corollary}\label{rem}

{\rm (i)} In the above theorem,  if $Y_0^{v^*}=\emptyset$ for each $v^* \in F^*$ (in particular, if
 the ranges of $T$ and $S$ contain constants), then  $Y_c$ is closed and $Y_d$ is open in $Y$.

{\rm (ii)} If $T$ is injective and $\cup_{f \in A(X,E)} \coz (Sf)=Y$ (in particular, if the range of $S$ contains at least one nonzero constant function), then $\phi$ has a dense range.
\end{corollary}
\begin{proof}
(i) Note that in this case we have $Y_c= \cap_{v^*\in F^*} Y^{v^*}_c  $. So the result followes from Lemma \ref{lem6}.

(ii) By hypotheses we have  $Y_1=\cup_{f \in A(X,E)} \coz (Tf)$. Now for each $x\in X$ and neighborhood  $V$ of $x$, choose, by regularity assumption on $A(X)$, a function $f\in A(X,E)$ with $f(x)\neq 0$, and $\supp(f) \subseteq V$. Since $f\neq 0$ and $T$ is injective, there exists $y\in Y$ such that $Tf(y)\neq 0$. Hence $y\in Y_1$. Since $Y_1= \cup_{v^*\in F^*} Y_1^{v^*}$, there exists $ v^*\in F^*$ such that 
\ref{coro1}, $\phi(y)=\phi_{v^*}(y) \in \supp(f) \subseteq V$. This
shows that $\phi(Y_1)$ is dense in $X$.
\end{proof}

\begin{theorem}\label{main2.2}
Let  $E$ be finite dimensional,  $A(X,E)$ be a subspace of $C(X,E)$ containing constants, $(A(X),\|\cdot\|_{A(X)})$ be a regular Banach function algebra on $X$ such that $A(X)\cdot A(X,E)
\subseteq A(X,E)$.  Assume, furthermore that there exists a
complete norm $\|\cdot\|$ on $A(X,E)$ with $\|f\|\ge
\|\|f\|_\infty$ such that $\|f \cdot g \|\leq \|g\|_{A(X)} \|f\|$
for each $f \in A(X,E)$ and $g\in A(X)$  and, moreover, $(A(X,E),
\|\cdot \|)$ have  (P) property.  Then for any jointly separating
linear operators $T, S: A(X,E) \longrightarrow C(Y,F)$
\begin{itemize}
\item[(i)] the maps $\Lambda^{T}$ and $\Lambda^{S}$ are continuous
functions from $Y_c$ into $L(E,F)$;

\item[(ii)] $\phi( \cup_{v^*\in F^*}Y_{d}^{v^*})$ is  finite and
consists of  limit points of $X$;

\item[(iii)] if one of $T$ or $S$ is injective and the
ranges of $S,T$ contain constant functions, then both $T$ and $S$ are continuous.
\end{itemize}
\end{theorem}

\begin{proof}
(i) First we  note that $\Lambda^{T}_y , \Lambda^{S}_y \in L(E,F)$ for each $y\in Y_c$, since $E$ is finite dimensional. Let $\{y_\alpha\}$ be a net in $Y_c$ converging to a point $y_0\in Y_c$. Then for any $e\in E$, by the definition of $\Lambda^{T}_y$ and continuity of $T\tilde{e}$, the net $\{\Lambda^{T}_{y_\alpha}(e)\}$ converges to $\Lambda^{T}_{y_0}(e)$, i.e $\| \Lambda^{T}_{y_\alpha}(e)- \Lambda^{T}_{y_0}(e)\| \to 0$. Since $E$ is of finite dimension, this easily implies that $\|\Lambda^{T}_{y_\alpha}-\Lambda^{T}_{y_0}\| \to 0$ in $L(E,F)$, that is  $\Lambda^{T}: Y_c \lo L(E,F)$ is continuous.  Continuity of  $\Lambda^{S}$ can be proven by a similar argument.

(ii) Put $Y_{d_1}= \cup_{v^*\in F^*}Y_{d}^{v^*}$ and assume, on the contrary, that there exists a sequence $\{y_n\}$ in $Y_{d_1}$ such that $\{\phi(y_n)\}$ is a sequence of distinct elements. $\Lambda^T_{y_n}(f_n(\phi(y_n)))$. Consider $g_n$ as an appropriate
multiple of $f_n-\widetilde {f_n(\phi(y_n)}$ such that $\|Tg_n(y_n)\|>n$. Clearly, $g_n(\phi(y_n))=0$. Now the subset $F_0=\{Tg_n(y_n): n\in \BN\}$ of $F$ is unbounded  (and consequently weakly unbounded)  so there exists  $v^*\in F^*$ such that $\sup\{|v^*(u)|: u\in F_0\}=\infty $. Thus $v^*(Tf_n(y_n))
\neq v^*( \Lambda^T_{y_n}(f_n(\phi(y_n)))$ for
 infinitely many $n\in \BN$, i.e., $\phi(Y_d^{v^*})$ is infinite, since for each $v^* \in F^*$ we have $v^*\circ \Lambda^{T}_y =\Lambda^{v^*,T}_y$. But this contradicts to Theorem \ref{main1} (iii). Finally $\phi(Y_d)$
 consists of limit points of $X$ since each $\phi(Y_d^{v^*})$ does.

(iii) We note that since the ranges of $T$ and $S$ contain constant functions, by Corollary \ref{rem}, we have $Y=Y_1=Y_d\cup Y_c$ and $Y_c$ is closed. Also, since for each $v^*\in F^*$ we have $Y_0^{v^*}=\varnothing$, one can see easily that $Y_d= \cup_{v^*\in F^*}Y_{d}^{v^*}=Y_{d_1}$. On the other hand, since $T$ or $S$ is injective, by Corollary \ref{rem},  the set $\phi(Y)=\phi(Y_d\cup Y_c)=\phi(Y_d)\cup \phi(Y_c)$ is dense in $X$. Since  $\phi(Y_d)$ is finite and consists of limit points of $X$, we conclude that $\phi(Y_c) $ is dense in $X$. Now the compactness of $Y_c$ and continuity of $\phi$ implies that $\phi(Y_c)=X$ and $Y_d=\emptyset$, that is $Y_c=Y$.

Finally, to prove that $T$ is continuous, let $\{f_n\}$ be a sequence in $A(X,E)$ such that $f_n\to f$ in $(A(X,E),\|\cdot\|)$ and $T(f_n) \to g$ in $C(Y,F)$. Since $\|\cdot \|_\infty\le
\|\cdot \|$, for each $y\in Y$ we have $f_n(\phi(y))\to
f(\phi(y))$ in $E$ and the continuity of  $\Lambda^T_y$ implies that, $\Lambda_y^T(f_n(\phi(y))\to \Lambda_y^T(f(\phi(y)))$ in $F$. On the other hand by (\ref{eq4}) we have $\Lambda_y^T(f_n(\phi(y)))=Tf_n(y)$ and $\Lambda_y^T(f(\phi(y)))=Tf(y)$. Hence $\|Tf_n(y)-Tf(y)\|\to 0$, for each $y\in Y$. This clearly implies that $Tf(y)=g(y)$ for each $y\in Y$. Since $A(X,E)$ is a Banach space, the Closed Graph Theorem implies that $T$ is continuous.

Continuity of $S$ will be proven in a similar way.
\end{proof}
\section{Applications of the results}
Let $X$ be a compact Hausdorff space and $E$ be a Banach space over $\Bbb K=\Bbb C \ or \ \Bbb R$. Clearly $C(X,E)$ itself satisfies the required conditions in Theorems \ref{main1} and
\ref{main2} and it also satisfy the hypotheses of Theorems
\ref{main1.1} and \ref{main2.2} whenever $E$ is finite
dimensional.

For applications of the results we give some examples of subspaces $A(X,E)$ satisfying the hypotheses of our main Theorems.

\begin{example}\label{exam}

(i) Let  $(X,d)$ be a compact metric space and $E$ be a Banach space. For $\alpha \in (0,1]$, we denote the space of $E$-valued Lipschitz functions of order $\alpha$ on $X$ by $\Lip^\alpha(X , E)$. That is, a continuous function $f : X \lo E$ belongs to $\Lip^\alpha(X,E)$ if and only if $$p_\alpha(f)= \sup \{  \frac{\|f(x) - f(y)\|}{d^\alpha(x,y)}: x,y \in X, x\neq y\}<\infty. $$ Then $\Lip^\alpha(X , E)$ is a Banach space with respect to the norm $\|f\|_\alpha=p_\alpha(f) +
\|f\|_\infty$, $f\in \Lip^\alpha(X,E)$.  The closed subspace
$\lip^\alpha(X , E)$ of $\Lip^\alpha(X , E)$ consists of those functions $f \in \Lip^\alpha(X , E)$  such that $$ \lim_{d(x,y)
\lo 0} \frac{\|f(x) - f(y)\|}{d^\alpha(x,y)}=0.$$ We write
$\Lip^\alpha(X)$ and $\lip^\alpha(X)$ for complex-valued case. One can see easily that $\Lip(X,E) \subseteq \lip^\alpha(X,E)$, for each $\alpha \in (0,1)$.

Let $0<\alpha<1$ and set $A(X,E)=\lip^\alpha(X,E)$ and $A(X)=\lip^\alpha(X)$. Then $(A(X),\|\cdot\|_\alpha)$ is a Banach function algebra on $X$ such that $A(X)\cdot A(X,E) \subseteq A(X,E)$ and $\|f\cdot g \|_{\alpha}\leq \|f\|_{\alpha}
\|g\|_{\alpha}$ for each $f\in A(X)$ and $g\in A(X,E)$.

For any compact subset $K$ of $X$ and open neighborhood $U$ containing it let the function $f_{K,U}$ on $X$ be defined by $f_{K,U}(x)=\frac{d(x, X\setminus U)}{d(x, K)+d(x, X\setminus U)}.$ Then by Lemma 2.2 in \cite{vw} $f_{K,U} \in \Lip(X)
\subseteq \lip^\alpha(X)$ with $ p_\alpha(f_{K,U}) \leq
\max\{\frac{1}{d(K, X\setminus U)}, 1\}$. Clearly $f_{K,U}=1$ on
$K$, $\coz(f_{K,U}) \subseteq U$ and $0\leq f_{K,U} \leq 1$. Hence $A(X)$ is a regular Banach function algebra on $X$.

We now show  that  $(A(X,E), \|\cdot \|_\alpha)$ satisfies Ditkin's condition with respect to $A(X)$. To do this,  let $f\in
\lip^\alpha(X, E)$ and $f(x_0)=0$ for some $x_0 \in X$. Let
$\varepsilon >0$ be arbitrary. Then there exists small enough $\delta >0$ such that for each $y \in X$ with $d(x_0, y) < \delta$ we have $\|f(y)\| < \epsilon$ and for each distinct points $x,y
\in X$ with  $d(x,y) <\delta$ we have $\frac{\|f(x) -
f(y)\|}{d^\alpha(x,y)}<\varepsilon$.  Let $V_\varepsilon$ and $U_\varepsilon$ be subsets of $X$ defined by $$V_\varepsilon =\{x\in X: d(x, x_0)<\frac{1}{4}\delta\},  \; {\rm and}\; \; U_\varepsilon =\{x \in X: d(x, x_0)<\frac{1}{2}\delta\}.$$ Setting $g_\varepsilon=1-f_{\overline{V_\varepsilon}, U_\varepsilon}$ we see that $ p_\alpha(g_\varepsilon) \leq
\max\{\frac{1}{d(\overline{V_\varepsilon}, X\setminus
U_\varepsilon)}, 1\}$, $g_\varepsilon=1$ on $\overline{V_\varepsilon}$, $\coz(g_\varepsilon) \subseteq U_\varepsilon$ and $0 \leq g_\varepsilon \leq 1$. A simple verification shows that $\|g_\varepsilon \cdot f -f\|_\alpha<5\varepsilon$.

\vspace*{.3cm} (ii) For the closed unit interval $\Bbb I=[0,1]$, a
function $f: \Bbb I\longrightarrow E$ is absolutely continuous if, for each $\varepsilon >0$, there exists $\delta >0$ such that $\Sigma_{i=1}^n\|f(t_i)-f(s_i)\| <\varepsilon$ whenever $\{(s_i,t_i)\}_{i=1}^n$ is a family of disjoint open subintervals of $\Bbb I$ with $\Sigma_{i=1}^n|t_i-s_i| < \delta $. We denote the set of absolutely continuous $E$-valued functions on $\Bbb I$ by ${\rm AC}(\Bbb I,E)$. The complex-valued case is denoted by ${\rm AC}(\Bbb I)$.  Given $f\in {\rm AC}(\Bbb I,E)$, we define the variation of $f$ on $I$ as $${\rm var}(f, \Bbb I)=\sup \{ \Sigma_{i=1}^n\|f(t_i)-f(t_{i-1})\| : 0=t_0 < t_1 < ... <t_n=1 \}.$$ Then  ${\rm AC}(\Bbb I,E)$ is a Banach space with respect to the norm $\| \cdot \|_{AC}$ defined by $$ \|f \|_{AC}= \|f\|_\infty + {\rm var}(f, I), \ \ \ \ \ \ (f \in {\rm AC}(\Bbb I,E)).$$ By \cite[Theorem 4.4.35]{d}, ${\rm AC}(I)$ is a regular Banach function algebra on $\Bbb I$ and it is easy to see that for each $f \in {\rm AC}(\Bbb I), g \in {\rm AC}(\Bbb I,E)$ we have $f \cdot g \in {\rm AC}(\Bbb I,E)$ and $\|f \cdot g
\|_{AC}\leq \|f \|_{AC} \|g \|_{AC} $. Following the same argument
as in \cite[Theorem 4.4.35(iii)]{d} we can easily see that $({\rm AC}(\Bbb I,E), \| \cdot \|_{AC})$ satisfies Ditkin's condition with respect to ${\rm AC}(\Bbb I)$. 

\vspace*{.3cm} (iii) Let $C^1(\Bbb I, E)$ be the space of all
$E$-valued continuously differentiable functions on the closed unit interval, that is the space of all functions $f : \Bbb I \lo E$ such that the derivative $f'(x)$, defined by $f'(x)=\lim _{h
\lo 0} \frac{f(x+h)-f(x)}{h}$, exists for each $x\in \Bbb I$, and
the function $f': \Bbb I \lo E$ is continuous. By \cite{adm}, the mean value Theorem is satisfied for the functions in $C^1(\Bbb I, E)$ so, a simple verification shows that $C^1(\Bbb I, E)$  is a Banach space with respect to the norm $\| \cdot \|$ defined by $$\|f \|= \|f\|_\infty + \|f'\|_\infty, \ \ \ \ \ (f \in C^1(\Bbb I, E)).$$ The complex-valued case is denoted by $C^1(\Bbb I)$. By \cite[Theorem 4.4.1]{d},  $C^1(\Bbb I)$, is a regular Banach function algebra. It is easy to see that for each $f \in C^1(\Bbb I)$ and $g \in C^1(\Bbb I, E)$ we have $f \cdot g \in C^1(\Bbb I, E)$. Since for any disjoint closed subsets $K$ and $F$ of $\Bbb I$, there exists $h\in C^1(\Bbb I)$ with $0\le h \le 1$, $h=1$ on $K$ and $h=0$ on $F$ it follows easily that for each $f\in C^1(\Bbb I,E)$ and $t_0\in \Bbb I$ with $f(t_0)=0$, there exists a sequence $\{f_n\}$ in $C^1(\Bbb I,E)$ such that each $f_n$ vanishes on a neighborhood of $t_0$ and $\|f_n-f\|_\infty
\to 0$ as $n \to \infty$. Hence $(C^1(\Bbb I,E),
\|\cdot\|_\infty)$ has property (P).
\end{example}

Using appropriate modifications in the definition of support points and the arguments in Sections 3 and 4 one can see that the same results hold for a pair of maps $T: A(X) \lo C(Y)$ and $S: A(X,E)\lo C(Y,F)$ satisfying
\[ f\cdot g=0 \Rightarrow Tf \cdot Sg=0 \;\;\;\; (f\in A(X), g\in
A(X,E)).\] Hence we have the next corollary.

\begin{corollary} Let $E$ and $F$ be real or complex Banach spaces, $A(X,E)$ be a subspace of $C(X,E)$ containing constants,
$A(X)$ be a regular Banach function algebra on $X$ such that $A(X)\cdot A(X,E) \subseteq A(X,E)$ and let $(A(X,E),
\|\cdot\|_\infty)$ and $(A(X), \|\cdot\|_\infty)$ have (P)
property. Let $T: A(X) \lo C(Y)$ and $S: A(X,E) \longrightarrow C(Y,F)$ be continuous real-linear jointly separating  maps such that the ranges of $T$ and $S$ contain constants. Then there exist a continuous map $\phi: Y \longrightarrow X$, a map $\Lambda^S: Y
\to L(E, F)$ continuous with respect to strong operator topology such that
\[Tf(y)= {\rm Re} (f(\phi(y)) \, T1(y)+ {\rm Im} (f(\Phi(y)))Ti(y)\;\;\;\;\;(f\in A(X), y\in Y) \]
and
\[ Sg(y)=\Lambda^S_y(g(\phi(y))\;\;\;\;\;\;(g\in A(X,E), y\in Y). \]
\end{corollary}

\begin{corollary}\label{coro}
Let $(A(X),\|\cdot\|_{A(X)})$  be a regular Banach function algebra on $X$ having (P) property. Let $A(Y)$ be a Banach function algebra on $Y$ and $T , S : A(X) \longrightarrow A(Y)$ be complex-linear jointly separating bijections. Then $X$ and $Y$ are homeomorphic and $T , S$ are jointly biseparating, that is $T^{-1} , S^{-1}$ are also jointly separating.
\end{corollary}
\begin{proof}
We first note that since $A(X)$ is a Banach function algebra, $\|f\|_\infty \leq \|f\|_{A(X)}$, for each $f\in A(X)$. Hence Theorems \ref{main2} and \ref{main2.2} hold with $E=F=\Bbb C$. Thus there exists a continuous function $\phi: Y_c \longrightarrow X$ such that
\begin{equation} \label{eq6}
Tf(y)=T1(y) f(\phi(y)) \ \ \ \ \  {\rm and} \ \ \ \ \ Sf(y)=S1(y) f(\phi(y));
\end{equation}

for each $f\in A(X)$ and $y\in Y_c$. Since $T, S$ are bijections, the first paragraph of the proof of \ref{main2.2} (iii) shows that $Y_c=Y$ and $\phi : Y \lo X$ is surjective.

To show that $T^{-1} , S^{-1}$ are jointly separating, let $Tf
\cdot Sg =0$ for some $f, g \in A(X)$. By (\ref{eq6}),   $T1$ and
$S1$ are nonvanishing functions, hence we have $(f\circ \phi) (g
\circ \phi) =0$ on $Y$. Since $\phi$ is surjective, $f \cdot g =0$
on $X$ that is, $T^{-1}$ and $S^{-1}$ are jointly separating.

Now we show that $\phi$ is a homeomorphism. To do so, it suffices to show that $\phi$ is injective. Let $y_1 , y_2$ be elements in $Y$ such that $\phi(y_1) = \phi(y_2)$. If $y_1 , y_2$ are distinct then we can find neighborhoods $V_{1}$ and $V_{2}$ of $y_1$ and $y_2$ respectively, with disjoint closures. For $i=1, 2$, let $g_i
\in A(Y)$ be such that $g_i(y_i)=1$ and ${\rm supp} g_i \subseteq
V_i$. Hence $g_1 \cdot g_2 =0$. Let $f_i\in A(X)$ be such that $Tf_1=g_1$ and $Sf_2=g_2$. Since $T^{-1}$, $S^{-1}$ are jointly separating we must have $f_1 \cdot f_2 =0$. On the other hand, by (\ref{eq6}) we have $$Tf_1(y_1)=T1(y_1) f_1(\phi(y_1)) \ \ \ \ \ {\rm  and} \ \ \ \ \ Sf_2(y_2)=S1(y_2) f_2(\phi(y_2))$$ which implies that $f_i(\phi(y_i))\neq 0$ since $g_i(y_i)\neq 0$, for $i=1, 2$, and $T1 , S1$ are nonvanishing functions. But this contradicts to the fact that $f_1 \cdot f_2 =0$. So $\phi$ is  injective.
\end{proof}

\end{document}